\documentclass[a4paper,12pt]{article}

\usepackage{amsmath,amsfonts,amssymb}
\usepackage{amsthm}
\usepackage{color,multicol}

\newtheorem{lemma}{Lemma}[section]
\newtheorem{theorem}[lemma]{Theorem}

\newtheorem{definition}[lemma]{Definition}

\begin{document}

\title{On estimates for the discrete eigenvalues of two-dimensional quantum waveguides}

\author{Martin Karuhanga\footnote{Department of Mathematics, Mbarara University of Science and Technology, Mbarara, Uganda. E-mail: \ mkaruhanga@must.ac.ug, ORCID: 0000-0002-7254-9073}\; and Catherine Ashabahebwa\footnote{Department of Mathematics, Kabale University, Kabale, Uganda. E-mail: cashabahebwa@kab.ac.ug, ORCID: 0009-0009-5298-2337}}

\date{}

\maketitle

\begin{abstract}
In this paper, we give upper estimates for the number and sum of discrete eigenvalues below the bottom of the essential spectrum  counting multiplicities of quantum waveguides in two dimensions.  We consider both straight and curved waveguides of constant width, and the estimates are presented in terms of norms of the potential. For the curved quantum waveguide, we assume that the waveguide is not self-intersecting and its curvature is a continuous and bounded function on $\mathbb{R}$. The estimates are new, particularly for the case of curved quantum waveguides.
\end{abstract}

\section{Introduction}
Given a non-negative function $V \in L^1_{\textrm{loc}}(\mathbb{R}^n)$, consider the self-adjoint operator on $L^2(\mathbb{R}^n)$
\begin{equation}\label{1}
 -\Delta - V, \;\;\;\;\;\;\;\; V \geq 0.
\end{equation}
The essential spectrum of \eqref{1} is the interval $[0, \infty)$ and its discrete spectrum consists of isolated negative eigenvalues of finite multiplicity \cite{GT}. A problem of physical interest is to obtain estimates for the number of these eigenvalues counting their multiplicities, and their sum in terms of regularity properties of the potential. For three and higher dimensions, an upper estimate for the number of negative eigenvalues is given by the well known Cwikel-Lieb-Rozenblum (CLR) inequality but this inequality is known to fail in two dimensions \cite{BE,Roz}. So far, the best estimates of the CLR type in two dimensions involve two terms; one in terms of the weighted $L^1$ norm and the other is given in terms of the Orlicz norms of the potential \cite{MK,Eugene, Sol}. Results similar to the CLR inequality for Schr\"{o}dinger operators with magetic field can be found for example in \cite{Bal, Sor}. The Lieb-Thirring (LT) inequalities give upper estimates for the sum of negative eigenvalues of \eqref{1} and they are known to have useful applications in the study of stability of matter \cite{Ed, Rup}. In the present paper, we present CLR and LT type inequalities for a straight and curved quantum waveguide in two dimensions with mixed Dirichlet and Neumann boundary conditions. Similar estimates of the CLR type have been presented in \cite{Kar1}. For studies relating to the existence of eigenvalues below the bottom of the essential spectrum for quantum waveguides, see for example \cite{Borr,ED,ES,Gol,Jil}. The results in this paper provide new estimates, particularly, in the case of curved quantum waveguides, and they have potential extensions to different configurations such as waveguides with local deformations by introducing impurities modelled by a Dirac perturbation \cite{EGST,EN,Kar1,MK}, waveguides with magnetic fields \cite{Bal,Sor} and waveguides coupled by a window \cite{Boris, EV}. \\\\
The operator studied in this paper is described below as follows:\\\\ Let $\Omega :=\mathbb{R}\times (0,d), d>0$ be an infinite two-dimensional strip and $V:\mathbb{R}^2\longrightarrow\mathbb{R}$ be a nonnegative  function integrable on bounded subsets of $\Omega$. We consider the motion of a quantum particle in $\Omega$ modelled by the operator
\begin{equation}\label{qp}
H:= -\Delta - V,\, V\ge 0
\end{equation}
on the Hilbert space $L^2(\Omega)$ subject to the boundary conditions
\begin{equation}\label{bcs}
u(x,0) = 0, \,\, \frac{\partial u}{\partial \textbf{n}}(x,d)=0,
\end{equation}
where $\textbf{n}$ denotes the outward normal unit vector. Because of the presence of the boundary conditions in \eqref{bcs}, the essential spectrum of the operator in \eqref{qp} is the interval $[\frac{\pi^2}{4d^2}, \infty)$. Therefore to estimate the number and sum of the eigenvalues below the bottom of the essential spectrum of \eqref{qp}, it is sufficient to obtain estimates for the number and sum of negative eigenvalues of the shifted operator
\begin{equation}\label{wg}
Q =-\Delta -\frac{\pi^2}{4d^2} - V,\, V\ge 0
\end{equation}
on $L^2(\Omega)$ with boundary conditions in \eqref{bcs}.\\\\
In obtaning estimates of the CLR type, the operator is decomposed into two independent operators. The first is defined by restricting the quadratic form associated with \eqref{wg} to the space of functions of the form $w(x)\sin\frac{\pi}{2d}y$ where $\sin\frac{\pi}{2d}y$ is the eigenfunction corresponding to the first eigenvalue $\lambda_1 = \frac{\pi^2}{4d^2}$ of the one-dimensional Laplacian on $L^2(0, d)$ subject to boundary conditions $u(0) = 0,\,\, \frac{du}{dy}(d) = 0$. The second one is defined by the class of functions orthogonal to $\sin\frac{\pi}{2d}y$ in the $L^2(0,d)$ inner product.

\section{Preliminaries}
\begin{definition}
{\em Let $(\Omega, \Sigma, \mu)$ be a measure space and let $\psi:[0,\infty)\longrightarrow[0,\infty)$ be a nondecreasing function. The class of measurable functions $f:\Omega\longrightarrow\mathbb{C}$ such that
$$
\int_{\Omega}\psi(|f(t)|)\,d\mu < \infty
$$
is called the Orcliz class $K_{\psi}(\Omega)$. The Orlicz space $L_{\psi}(\Omega)$ is the linear span of $K_{\psi}(\Omega)$. }
\end{definition}
\begin{definition}
{\em A continuous nondecreasing convex function  $\psi:[0,\infty)\longrightarrow[0,\infty)$ is called an $N$-function if 
$$
\underset{t\longrightarrow 0^+}\lim\frac{\psi(t)}{t} = 0 \mbox{ and } \underset{t\longrightarrow \infty}\lim\frac{\psi(t)}{t} = \infty.
$$
The function $\phi:[0,\infty)\longrightarrow[0,\infty)$  given by $\phi(t) = \underset{s\ge 0}\sup(st-\psi(s))$ is called complementary to $\psi$.}
\end{definition}
Let $\phi$ and $\psi$ be mutually complementary $N$-functions and let $L_{\phi}(\Omega)$ and $L_{\psi}(\Omega)$ be the corresponding Orlicz spaces. Then the following norms are equivalent on $L_{\psi}(\Omega)$.
$$
\|f\|_{\psi,\Omega} = \sup\left\{
\left|\int_{\Omega}fgd\mu\right|\,:\,\int_{\Omega}\phi gd\mu \le 1\right\}
$$
and 
$$
\|f\|_{(\psi,\Omega)} = \inf\left\{
\kappa\,:\,\int_{\Omega}\phi\left(\frac{f}{\kappa}\right)d\mu \le 1\right\}.
$$
If $\mu(\Omega) <\infty$, the following norm is equivalent to the norm on $L_{\psi}(\Omega)$\cite{Sol}.
\begin{equation}\label{equi}
\|f\|^{(av)}_{\psi,\Omega} =\sup\left\{
\left|\int_{\Omega}fgd\mu\right|\,:\,\int_{\Omega}\phi( g)d\mu \le \mu(\Omega)\right\}. 
\end{equation}
For our purpose, we shall use the following pair of $N$-complementary functions
\begin{equation}\label{N}
\phi(s) = e^{|s|}-1-|s|,\,\, \psi(t)= (1+|t|)\ln(1+|t|)-|t|,\,s,t\in\mathbb{R}.
\end{equation}
\begin{definition}
{\em  Let $I_1,I_2\subset\mathbb{R}$ be nonempty intervals. We shall define the space $L_1(I_1, L_{\psi}(I_2))$ to be the space of functions $f:I_1\times I_2\longrightarrow\mathbb{R}$ such that
\begin{equation}\label{norm}
\|f\|_{L_1(I_1, L_\psi(I_2))} := \int_{I_1}\|f\|^{(av)}_{\psi,I_2}dx < \infty.
\end{equation}}
\end{definition}

Let $\mathcal{H}$ be a Hilbert space and $q$ be a sesquilinear form with domain $D(q)\subset\mathcal{H}$. Let 
\begin{equation}\label{number}
N_-(q) = \sup\{\mbox{ dim } M\,:\,q[u]\le 0\, \forall u\in M\setminus{0}\},
\end{equation}
where $M$ is a linear space of $D(q)$. If $q$ is the quadratic form of a lower semi-bounded self-adjoint operator $A$ with no essential spectrum in the interval $(-\infty, 0)$, then \eqref{number} gives the number of negative eigenvalues of $A$ counting their multiplicities \cite{BirSol}.

\section{Estimates for straight quantum waveguides}
The study of eigenvalues properties of straight quantum waveguides with mixed boundary conditions has been extensive in recent years, see for example \cite{Ditt, Jil}. The presence of boundary conditions produces isolated eigenvalues below the bottom of the essential spectrum, therefore providing estimates for their number is essential for physical applications. In this section, we determine upper estimates in presence of Dirichlet and Neumann boundary conditions. Similar estimates can be found for example in \cite{Grig,Kar1}.\\\\
The quadratic form associated with \eqref{wg}  is given by
\begin{equation}\label{quad}
q[u] = \int_{\Omega}(|\nabla u|^2 -\frac{\pi^2}{4d^2}|u|^2)dxdy - \int_{\Omega}V|u|^2dxdy
\end{equation} with domain
$$
D(q) = \{u\in W^1_2(\Omega, dxdy)\,:\,u(x, 0) =0 ,\,\frac{\partial u}{\partial \textbf{n}}(x,d)=0\}\cap L^2(\Omega, Vdxdy).
$$
Let $\mathcal{H} := \{u\in W^1_2(\Omega, dxdy)\,:\,u(x, 0) =0 ,\,\frac{\partial u}{\partial \textbf{n}}(x,d)=0\}$ and define\\ $\mathcal{H}_1 = P\mathcal{H}$ and $\mathcal{H}_2 =(I- P)\mathcal{H}$, where
\begin{equation}\label{proj}
Pu(x,y) = \left(\frac{2}{d}\int_0^du(x,y)\sin\frac{\pi}{2d}y dy\right)\sin\frac{\pi}{2d}y = h(x)\sin\frac{\pi}{2d}y,\, \forall u\in\mathcal{H},
\end{equation} where
$$
h(x):= \frac{2}{d}\int_0^du(x,y)\sin\frac{\pi}{2d}y dy.
$$
Then $P$ is a projection. Below we give auxilliary results necessary for the proof of the main result of this section.\\\\
\begin{lemma}\label{lemma1}
 Let $\oplus$ denote the direct orthogonal sum. Then $\mathcal{H} = \mathcal{H}_1\oplus\mathcal{H}_2$.
\end{lemma}
\begin{proof}
We need to show that $\langle f, g\rangle =0$ and $\langle \nabla f,\nabla g\rangle =0$ in $L^2(\Omega)$ for all $f\in\mathcal{H}_1$ and for all $g\in\mathcal{H}_2$.
\begin{eqnarray*}
 \langle f, g\rangle &=& \int_{\Omega}h(x)\sin\frac{\pi}{2d}y(I-P)\overline{u(x,y)}\,dxdy\\ &=& \int_{\mathbb{R}}h(x)\left(\int_0^d\overline{u(x,y)}\sin\frac{\pi}{2d}y\,dy-\left(\frac{2}{d}\int_0^d\overline{u(x,y)}\sin\frac{\pi}{2d}y\,dy\right)\int_0^d\sin^2\frac{\pi}{2d}y\,dy
 \right)\,dx\\ &=& 0.\\
\langle \nabla_x f, \nabla_x g\rangle &=& \int_{\Omega}\partial_x h(x)\sin\frac{\pi}{2d}y(I-P)\partial_x\overline{u(x,y)}\,dxdy\\
&=&\int_{\mathbb{R}}\partial_x h(x)dx\int_0^d\partial_x\overline{u(x,y)}\sin\frac{\pi}{2d}y\,dy\\ &-&\left(\frac{2}{d}\int_0^d\partial_x\overline{u(x,y)}\sin\frac{\pi}{2d}y\,dy\right)\int_0^d\sin^2\frac{\pi}{2d}y\,dy
 \\ &=&0.
\end{eqnarray*}
Similarly, using integration by parts and applying the boundary conditions in \eqref{bcs}, one gets $\langle \nabla_y f, \nabla_y g\rangle = 0$.
\end{proof}
\begin{lemma}\label{lemma2}
 The subspace $\mathcal{H}_2$ is orthogonal to $\sin\frac{\pi}{2d}y$ in the $L^2((0,d))$ inner product.
\end{lemma}
\begin{proof}
For all $g\in\mathcal{H}_2$, we have
\begin{eqnarray*}
\langle g, \sin\frac{\pi}{2d}y\rangle &=& \int_{0}^d(I-P)u(x,y)\sin\frac{\pi}{2d}y\,dy\\
&=&\int_0^du(x,y)\sin\frac{\pi}{2d}y\,dy -\frac{2}{d}\int_0^d\sin^2\frac{\pi}{2d}ydy\int_0^du(x,y)\sin\frac{\pi}{2d}ydy\\ &=& 0.
\end{eqnarray*}
\end{proof}

It follows from lemma \ref{lemma1} that for all $u\in\mathcal{H}$, $u = v+w,\,v\in\mathcal{H}_1,\,w\in\mathcal{H}_2$. Thus, one has
$$
\int_{\Omega}(|\nabla u|^2-\frac{\pi^2}{4d^2}|u|^2)dxdy = \int_{\Omega}(|\nabla v|^2-\frac{\pi^2}{4d^2}|v|^2)dxdy + \int_{\Omega}(|\nabla w|^2-\frac{\pi^2}{4d^2}|w|^2)dxdy
$$
and
$$
\int_{\Omega}V|u|^2dxdy = \int_{\Omega}V|v+w|^2dxdy \le 2\int_{\Omega}V|v|^2dxdy + 2\int_{\Omega}V|w|^2dxdy.
$$
Hence the quadratic form in \eqref{quad} gives
\begin{eqnarray*}
q[u]&\ge& \int_{\Omega}(|\nabla v|^2-\frac{\pi^2}{4d^2}|v|^2)dxdy - 2\int_{\Omega}V|v|^2dxdy\\ &+& \int_{\Omega}(|\nabla w|^2-\frac{\pi^2}{4d^2}|w|^2)dxdy - 2\int_{\Omega}V|w|^2dxdy.
\end{eqnarray*}

Let 
$$
q_1[v]:= \int_{\Omega}(|\nabla v|^2-\frac{\pi^2}{4d^2}|v|^2)dxdy - 2\int_{\Omega}V|v|^2dxdy,
$$
$$
D(q_1) =  \{v\in\mathcal{H}_1\,:\,v(x, 0) =0 ,\,\frac{\partial v}{\partial \textbf{n}}(x,d)=0\}\cap L^2(\Omega, Vdxdy)
$$ and
$$
q_2[w]:= \int_{\Omega}(|\nabla w|^2-\frac{\pi^2}{4d^2}|w|^2)dxdy - 2\int_{\Omega}V|w|^2dxdy,
$$
$$
D(q_2) = \{v\in\mathcal{H}_2\,:\,w(x, 0) =0 ,\,\frac{\partial w}{\partial \textbf{n}}(x,d)=0\}\cap L^2(\Omega, Vdxdy).
$$
Then
\begin{equation}\label{ineq}
q[u]\ge q_1[v] + q_2[w].
\end{equation}

If $N_-(q)$ denotes the number of isolated eigenvalues in the interval $(-\infty, 0)$ of the operator \eqref{wg} counting their multiplicities, then 
\begin{equation}\label{Number1}
N_-(q) \le N_-(q_1) + N_-(q_2).
\end{equation}

Define a sequence of partitions $\{\Omega_k\}_{k\in\mathbb{N}}$ of $\Omega$ by $\Omega_k := (k, k+1)\times(0,d)$ and consider the following eigenvalue problem on $L^2(\Omega_k)$;
\begin{equation*}
\begin{cases}
-u'' = \lambda u\;\;\;\;\; & \textrm{in}\;\; \Omega_k,\\
u(x,0)= 0,\\
\frac{\partial u}{\partial \textbf{n}}(x,d)=0.
\end{cases}
\end{equation*}
Let $\lambda_1\le\lambda_2\le ...\le\lambda_n\le ...$ be the eigenvalues of the above boundary value problem. Then it is easy to find that $\lambda_1 = \frac{\pi^2}{4d^2}, \,\lambda_2 = \mbox{ min }\{\pi^2 + \frac{\pi^2}{4d^2}, \frac{9\pi^2}{4d^2}\} > \lambda_1 $ so that $\lambda_2 - \lambda_1 = \pi^2\mbox{ min }\{1, \frac{2}{d^2}\}$.
The Min-Max principle implies that for all $u\in W^1_2(\Omega_k)$ with $u\perp\sin\frac{\pi}{2d}y$, we have
$$
\lambda_2\int_{\Omega_k}|u|^2dxdy \le \int_{\Omega_k}|\nabla u|^2dxdy.
$$
This in turn implies
\begin{eqnarray*}
\int_{\Omega}(|\nabla u|^2-\lambda_1|u|^2)dxdy &=& \int_{\Omega}(|\nabla u|^2-\lambda_2|u|^2)dxdy +(\lambda_2 - \lambda_1)\int_{\Omega}| u|^2dxdy\\
&\ge&(\lambda_2 - \lambda_1)\int_{\Omega}| u|^2dxdy\\ &=& \pi^2\mbox{ min }\left\{1, \frac{2}{d^2}\right\}\int_{\Omega}| u|^2dxdy.
\end{eqnarray*}
Hence
\begin{eqnarray}\label{min}
\int_{\Omega}| u|^2dxdy &\le& \frac{1}{\pi^2}\mbox{ max }\left\{1, \frac{d^2}{2}\right\}\left(\int_{\Omega}(|\nabla u|^2-\frac{\pi^2}{4d^2}|u|^2)dxdy\right)\nonumber\\
&=&C_1\int_{\Omega}(|\nabla u|^2-\frac{\pi^2}{4d^2}|u|^2)dxdy,
\end{eqnarray}

for all $u\in W^1_2(\Omega_k)$ with $u\perp\sin\frac{\pi}{2d}y$, where $C_1 = \frac{1}{\pi^2}\mbox{ max }\left\{1, \frac{d^2}{2}\right\}$.

\begin{lemma}\label{lemma3}
There exists a constant $C_2>0$ such that
$$
\int_{\Omega_k}| u|^2dxdy \le C_2\left(\int_{\Omega_k}(|\nabla u|^2-\frac{\pi^2}{4d^2}|u|^2)dxdy\right)
$$
for all $u\in W^1_2(\Omega_k)$ with $u\perp\sin\frac{\pi}{2d}y$.
\end{lemma}
\begin{proof}
It follows from \eqref{min} that for all $u\in W^1_2(\Omega_k)$ with $u\perp\sin\frac{\pi}{2d}y$
\begin{eqnarray*}
\int_{\Omega_k}(|\nabla u|^2 dxdy&=& \int_{\Omega_k}(|\nabla u|^2-\frac{\pi^2}{4d^2}|u|^2)dxdy + \int_{\Omega_k}\frac{\pi^2}{4d^2}|u|^2)dxdy\\
&\le& \max\left\{1,\frac{C_1\pi^2}{4d^2}\right\}\int_{\Omega}(|\nabla u|^2-\frac{\pi^2}{4d^2}|u|^2)dxdy.
\end{eqnarray*} 
\end{proof}
\begin{lemma}{\rm \cite[Lemma 4.3]{MK}}\label{lemma4}
 Let $J_k:= (k, k+1), k\in\mathbb{N}$ and $I:= (0,d)$. Then for any $V\ge 0$ and any $m\in\mathbb{N}$, there exists a finite cover of $\Omega_k$ by rectangles $R_{kj},\,j=1,2, ..., m_0$ such that $m_0< m$ and 
\begin{equation}\label{m}
\int_{\Omega_k}V|w|^2dxdy \le C_3m^{-1}\|V\|_{L_1(J_k, L_{\psi}(I))}\int_{\Omega_k}(|\nabla w|^2 + |w|^2)dxdy
\end{equation}
for all $w\in W^1_2(\Omega_k)$ such that $w(x,0) =0,\, w'(x,d)= 0$ with $\frac{1}{|R_{ij}|}\int_{R_{ij}}w dxdy =0$ and $C_3$ is a constant independent of $V$.
\end{lemma}

Define a partition of $\mathbb{R}$ by the intervals
$$
I_k = [2^{k-1}, 2^k],\,k>0,\,I_0=[-1,1],\, I_k = [2^{|k|}, 2^{|k|-1}],\,k<0.
$$ Define the following quantities;
\begin{eqnarray*}
\beta_0 &:=& \frac{2}{d}\int_{I_0}\left(\int_0^dV\sin^2\frac{\pi}{2d}y\,dy\right)\,dx,\\
\beta_k &:=& \frac{2}{d}\int_{I_k}\left(\int_0^d|x|V\sin^2\frac{\pi}{2d}y\,dy\right)\,dx, k\not= 0,\\
\mathcal{C}_k &:=& \|V\|_{L_1(J_k, L_{\psi}(I))}.
\end{eqnarray*}

\begin{theorem}\label{resut1}
For any $V\ge 0$ there are constants $C_4,C_5, C_6>0$ such that
\begin{equation}\label{mr1}
N_-(q)\le 1+ C_6\left(\sum_{\beta_k >C_4}\sqrt{\beta_k}+ \sum_{\mathcal{C}_k>C_5}\mathcal{C}_k\right).
\end{equation}
\end{theorem}
\begin{proof}
Since $q_1$ is the restriction of the quadratic form $q$ to the subspace $\mathcal{H}_1$, we have
\begin{eqnarray*}
q_1[v] &=& \int_{\Omega}|h'(x)\sin\frac{\pi}{2d}|^2 dxdy +\frac{\pi}{2d} \int_{\Omega}|h(x)\cos\frac{\pi}{2d}|^2 dxdy - \frac{\pi^2}{4d^2} \int_{\Omega}|h(x)\sin\frac{\pi}{2d}|^2 dxdy\\ &-& 2\int_{\Omega}V|h(x)\sin\frac{\pi}{2d}|^2 dxdy\\ &=&\int_{\mathbb{R}}|h'(x)|^2dx\int_0^d\sin^2\frac{\pi}{2d}ydy + \frac{\pi^2}{4d^2}\int_{\mathbb{R}}|h(x)|^2dx\int_0^d\cos^2\frac{\pi}{2d}ydy\\ &-& \frac{\pi^2}{4d^2}\int_{\mathbb{R}}|h(x)|^2dx\int_0^d\sin^2\frac{\pi}{2d}ydy -2\int_{\mathbb{R}}|h(x)|^2dx\int_0^dV\sin^2\frac{\pi}{2d}ydy\\ &=& \frac{d}{2}\int_{\mathbb{R}}|h'(x)|^2dx - 2\int_{\mathbb{R}}|h(x)|^2dx\int_0^dV\sin^2\frac{\pi}{2d}ydy\\
&=&\frac{d}{2}\left(\int_{\mathbb{R}}|h'(x)|^2dx - 2\int_{\mathbb{R}}\hat{V}|h(x)|^2dx\right)
\end{eqnarray*} where 
$$
\hat{V}(x)= \frac{2}{d}\int_0^d V\sin^2\frac{\pi}{2d}y\,dy.
$$
Therefore the form $q_1$ is a quadratic form associated with the one-dimensional self-adjoint operator
$$
-\frac{d^2}{dx^2} - 2\hat{V} \mbox{ on } L^2(\mathbb{R}).
$$ 
It follows from the definitions of $\beta_0$ and $\beta_k$ that
$$
\beta_0 = \int_{I_0}\hat{V}(x)\,dx \mbox{ and } \beta_k = \int_{I_k}|x|\hat{V}(x)\,dx, \,k\not= 0.
$$Hence there is a constant $ C_7>0$ such that
\begin{equation}\label{est1}
N_-(q_1) \le 1+ C_7\sum_{\beta_k >C_4}\sqrt{\beta_k},
\end{equation}
see for example  \cite{Grig,Eugene,Sol}. It remains to obtain the last term in the right hand side of \eqref{mr1}. This is achieved by restricting the form $q$ to the subspace $\mathcal{H}_2$. Lemma \ref{lemma3} and \eqref{min} imply that for all $w\in W^1_2(\Omega_k)$ with $w\perp \sin\frac{\pi}{2d}y$
$$
\int_{\Omega_k}(|\nabla w|^2 +|w|^2)dxdy \le C_8\left(\int_{\Omega_k}(|\nabla w|^2 -\frac{\pi^2}{4d^2}|w|^2)dxdy\right),
$$
where $C_8 = \max\{C_1, C_2\}$. By Lemma \ref{lemma4}, we have that for all $w\in W^1_2(\Omega_k)\cap C(\overline{\Omega_k})$ satisfying the $m_0$ orthogonality conditions
$$
2\int_{\Omega_k}V|w|^2dxdy \le C_9m^{-1}\|V\|_{L_1(J_k, L_{\psi}(I))}\left(\int_{\Omega_k}(|\nabla w|^2 -\frac{\pi^2}{4d^2}|w|^2)dxdy\right),
$$
where $C_9 = 2C_3C_8$. Let $m = \lceil C_9\|V\|_{L_1(J_k, L_{\psi}(I))}\rceil + 1$, where $\lceil  x\rceil$ denotes the largest integer not greater than $x$, then similarly to 
{\rm \cite[Lemma 3.2.14]{MK}}, we have
\begin{equation}\label{est2}
N_-(q_2) \le C_9\|V\|_{L_1(J_k, L_{\psi}(I))} + 2.
\end{equation}
Again, taking $m = 1$, we have 
$$
2\int_{\Omega_k}V|w|^2dxdy \le C_9\|V\|_{L_1(J_k, L_{\psi}(I))}\left(\int_{\Omega_k}(|\nabla w|^2 -\frac{\pi^2}{4d^2}|w|^2)dxdy\right).
$$
If $\|V\|_{L_1(J_k, L_{\psi}(I))} \le \frac{1}{C_9}$, then
$$
2\int_{\Omega_k}V|w|^ddxdy \le \int_{\Omega_k}(|\nabla w|^2 -\frac{\pi^2}{4d^2}|w|^2)dxdy.
$$
This implies that $N_-(q_2)$ = 0. Otherwise, \eqref{est2} gives
\begin{equation}\label{eq}
N_-(q_2) \le 2\left(\frac{1}{2}C_9\|V\|_{L_1(J_k, L_{\psi}(I))} + 1\right) = C_{10}\mathcal{C}_k,
\end{equation} where $C_{10} = 3C_9$. Hence for all $C_5< \frac{1}{C_{10}}$, it follows by {\rm \cite[Lemma 1.6.2]{MK}} that
\begin{equation}\label{est3}
N_-(q_2) \le C_{10}\sum_{\mathcal{C}_k>C_5}\mathcal{C}_k,
\end{equation}
Hence \eqref{mr1} follows from \eqref{Number1}, \eqref{est1} and \eqref{est3}, where $C_6 = \max\{C_7, C_{10}\}$.
\end{proof}

Let $\alpha >0$. Then scaling the potential $V$ by $\alpha$ gives the following semi-classical growth for $N_-(q)$ 
\begin{equation}\label{classical}
N_-(q) = O(\alpha) \mbox{ as } \alpha \longrightarrow\infty,
\end{equation}
see e.g.,\cite{Eugene} and the references therein. One can easily show that the finiteness of the first term in \eqref{mr1} is necessary for \eqref{classical} to hold (cf.{\rm \cite[Theorem 6]{Kar1}}). However, finiteness of the second term does not necessarily mean that \eqref{mr1} satisfies \eqref{classical} and vice versa {\rm \cite[Section 9 ]{Eugene}}. On the other hand, no estimate of the form in \eqref{mr1} can hold for any Orlicz norm weaker than the norm of $V$ in the second term {\rm \cite[Theorem 9.4]{Eugene}}.

\section{Estimates for curved quantum waveguides}
The spectral properties of curved quantum waveguides with various boundary conditions have been studied by different authors. It is well known that the presence of eigenvalues below the bottom of the essential spectrum depends on both the geometry of the waveguide and choice of boundary conditions. More information regarding these results can be found in \cite{ES,ET, KJ} and the references therein. 
\subsection{The configuration space and the Laplacian}
Let $\gamma$ be a unit speed curve , that is, a $C^2$ smooth embedding $\gamma:\mathbb{R}\longrightarrow\mathbb{R}^2: \{s \longmapsto (\gamma_1(s), \gamma_2(s))\}$ satisfying $|\dot{\gamma}(s)|=1$ for all $s\in\mathbb{R}$. The normal unit vector field is defined by $N = (-\dot{\gamma_2}, \dot{\gamma_1})$ and $T = (\dot{\gamma_1}, \dot{\gamma_2})$ gives the tangent field. The pair (T, N) gives the Frenet-Serret frame \cite{Klin}. The curvature $k$ of $\gamma$ is defined through the Frenet-Serret formulae by $k(s)= \mbox{ det }(\dot{\gamma}, \ddot{\gamma})$ as a continuous function of the arc-length parameter $s$. We define a curved strip of constant width $d$ as $\Omega' = L(\Omega)$, where 
$$
L:\mathbb{R}^2\longrightarrow\mathbb{R}^2:\{(s,u)\longmapsto \gamma(s)+ uN(s)\}.
$$
We shall make the following assumption:\\\\ 
(A): $\Omega'$ is non-self intersecting and $k\in L^{\infty}(\mathbb{R})$ with $d\|k_+\|_{\infty} < 1$, where $k_+ = \max\{0, \pm k\}$.\\\\
The above assumption implies that $s\longmapsto L(s,u)$ for a fixed $u\in (0,d)$ traces out a parallel curve at a distance $|u|$ from $\gamma$ and $u\longmapsto L(s,u)$ for a fixed $s\in\mathbb{R}$ is a straight line orthogonal to $\gamma$ at $s$. In addition, the mapping $L: \Omega \longrightarrow\Omega'$ is a $C^1$-diffeomorphism and its inverse determines a system of coordinates $(s,u)$ in the neighbourhood of $\gamma$. Through the Frenet-Serret formulae, the metric tensor $G$ of $\Omega'$ is given by
\begin{equation}\label{tensor}
G(s,u)=
\begin{pmatrix}
(1-uk(s))^2 & 0\\
0 & 1
\end{pmatrix}
\end{equation}

The determinant of $G$ is given by $|G| = (1-uk(s))^2$ and its inverse is
$$
G^{-1}(s,u)=
\begin{pmatrix}
\frac{1}{1-uk(s)^2} & 0\\
0&1
\end{pmatrix}.
$$
The area of the element of the curved strip is defined through
$$
d\Omega' = \sqrt{G}dsdu = (1-uk(s))dsdu.
$$
By the virtue of the second part of assumption (A), it is clear that the metric tensor $G$ in \eqref{tensor} is uniformly elliptic. In particular, for all $(s,u)\in\Omega$, we have
\begin{equation}\label{elliptic}
1-d\|k_+\|_{\infty}\le 1-uk(s)\le 1+d\|k_-\|_{\infty}.
\end{equation}
The expression for the Laplacian on the strip equipped with the metric tensor \eqref{tensor} is given by
\begin{eqnarray}\label{Lapl}
-\Delta_{\Omega} &=& -\frac{1}{\sqrt{G}}(\partial_s, \partial_u)\left(\sqrt{G}G^{-1}\begin{pmatrix}\partial_s\\\partial_u \end{pmatrix}\right)\nonumber\\
&=&-\frac{\partial_s^2}{(1-uk(s))^2} - \frac{u\dot{k}(s)}{(1-uk(s))^3}\partial_s -\partial_u^2 + \frac{k(s)}{1-uk(s)}\partial_u.
\end{eqnarray}
The domain of \eqref{Lapl} is the subspace of the $W^2_2(\Omega)$ whose elements satisfy the boundary conditions \eqref{bcs}\cite{Krez}.
The Schr\"{o}dinger operator in \eqref{wg}  on the strip with metric tensor $G$ becomes
\begin{equation}\label{curvedop}
Q' = -\Delta_{\Omega}-\frac{\pi^2}{4d^2} - V, \,V\ge 0
\end{equation}
subject to the boundary conditions in \eqref{bcs}. If $q'$ is the quadratic form corresponding to $Q'$, then for every $f\in L^2(\Omega, (1-uk(s)dsdu))$, we have
\begin{eqnarray}\label{newform}
q'[f] &=& \int_{\Omega}\frac{|\partial_sf|^2}{1-uk(s)}dsdu + \int_{\Omega}|\partial_uf|^2(1-uk(s))dsdu - \frac{\pi^2}{4d^2}\int_{\Omega}|f|^2(1-uk(s))dsdu\nonumber\\
&-&\int_{\Omega}V|f|^2(1-uk(s))dsdu,\\
D(q')&=&\{f\in W^1_2(\Omega, (1-uk(s)dsdu):f(x,0) = 0, \, \frac{\partial f}{\partial \textbf{n}}(x,d)=0\}\cap L^2(\Omega, V(1-uk(s)dsdu).\nonumber
\end{eqnarray}
\subsection{Estimates for the number of eigenvalues}
By the ellipicity property in \eqref{elliptic}, it follows from \eqref{newform} that
\begin{eqnarray*}
q'[f]&\ge&\frac{1}{1+d\|k_-\|_{\infty}}\int_{\Omega}|\partial_sf|^2dsdu + (1-d\|k_+\|_{\infty})\int_{\Omega}|\partial_uf|^2dsdu \\ &-&\frac{\pi^2}{4d^2}(1+d\|k_-\|_{\infty})\int_{\Omega}|f|^2dsdu - (1+d\|k_-\|_{\infty})\int_{\Omega}V|f|^2dsdu  \\
&\ge& \min\left\{\frac{1}{1+d\|k_-\|_{\infty}},1-d\|k_+\|_{\infty}\right\}\left(\int_{\Omega}|\partial_sf|^2dsdu + \int_{\Omega}|\partial_uf|^2dsdu\right)\\
&-&\frac{\pi^2}{4d^2}(1+d\|k_-\|_{\infty})\int_{\Omega}|f|^2dsdu - (1+d\|k_-\|_{\infty})\int_{\Omega}V|f|^2dsdu
\end{eqnarray*}
Let
$$
m:=\min\left\{\frac{1}{1+d\|k_-\|_{\infty}},1-d\|k_+\|_{\infty}\right\}.
$$
Then
\begin{eqnarray}\label{curveform}
q'[f]&\ge& m\left(\int_{\Omega}|\partial_sf|^2dsdu + \int_{\Omega}|\partial_uf|^2dsdu
-\frac{\pi^2}{4md^2}(1+d\|k_-\|_{\infty})\int_{\Omega}|f|^2dsdu\right)\nonumber\\ &-&\frac{1}{m} (1+d\|k_-\|_{\infty})\int_{\Omega}V|f|^2dsdu.
\end{eqnarray}
The right hand side of \eqref{curveform} represents a quadratic form of a lower semi-bounded self-adjoint operator
\begin{equation}\label{newoperator}
-\Delta - \lambda'_1-V',\,\,V'\ge 0
\end{equation}
on $L^2(\Omega, dsdu)$, where $\lambda'_1= \frac{\pi^2}{4md^2}(1+d\|k_-\|_{\infty})$ and $V'= \frac{V(1+d\|k_-\|_{\infty})}{m}$.\\
Let  $q'_0$ be the quadratic form corresponding to \eqref{newoperator}. Then
$$
q'[f] \ge q_0'[f].
$$
Hence, it follows by \eqref{number} that
\begin{equation}\label{est5} 
N_-(q')\le N_-(q'_0).
\end{equation}

The transverse eigenfunction corresponding to $\lambda'_1$ is given by\\ $\eta(u) = \sin \frac{\pi l}{2d}u$, where
$$
l = \sqrt{\frac{1+d\|k_-\|_{\infty}}{m}}.
$$
Let
$$
\beta := \int_0^d\sin^2 \frac{\pi l}{2d}u\,du
$$
and let $\mathcal{H}'$ denote the domain of $q'_0$. Define $P:\mathcal{H}'\longrightarrow\mathcal{H}'_1 $ by
\begin{equation}\label{proj}
Pf(s,u) = \left(\frac{1}{\beta}\int_0^df(s,u)\sin\frac{\pi l}{2d}y du\right)\sin\frac{\pi l}{2d}u = j(s)\sin\frac{\pi l}{2d}u
\end{equation} for all $f\in\mathcal{H}'$, where
$$
j(s):= \frac{1}{\beta}\int_0^df(s,u)\sin\frac{\pi l}{2d}u\, du.
$$
Then $P$ is a projection. Let $\mathcal{H}'_2 = (I-P)\mathcal{H}'$, then similarly to lemmas \ref{lemma1} and \ref{lemma2}, we can easily show that $\mathcal{H}' = \mathcal{H}'_1\oplus\mathcal{H}'_2$, and that the subspace $\mathcal{H}'_2$ is orthogonal to $\sin \frac{\pi l}{2d}u$ in $L^2(0,d)$ respectively.\\\\

Let the intervals $I_k$ and $J_k$ be defined as before. Let
\begin{eqnarray*}
\gamma_0 &:=& \frac{l^2}{\beta}\int_{I_0}\left(\int_0^dV(s,u)\sin^2\frac{\pi l}{2d}u\,du\right)\,ds,\\
\gamma_k &:=& \frac{l^2}{\beta}\int_{I_0}\left(\int_0^d|s|V(s,u)\sin^2\frac{\pi l}{2d}u\,du\right)
\,ds,\,k\not=0,\\
\mathcal{D}_k &:=&\|V\|_{L_1(J_k, L_{\psi}(I))}.
\end{eqnarray*}

\begin{theorem}\label{mr2}
 For any $V\ge 0$ there exist constants $C_{11}, C_{12}>0$ such that
\begin{equation}\label{est4}
N_-(q') \le 1+ C_{13}\left(\sum_{\gamma_k > C_{11}}\sqrt{\gamma_k} + \sum_{\mathcal{D}_k>C_{12}}\mathcal{D}_k\right),
\end{equation}
where $C_{13}$ is a constant depending on the curvature $k$.
\end{theorem}
\begin{proof}
If $q'_{01}$ and $q'_{02}$ are the restrictions of the form $q'_0$ to the subspaces $\mathcal{H}'_1$ and $\mathcal{H}'_2$ respectively, then
\begin{equation}\label{est13}
N_-(q'_0) \le N_-(q'_{01}) + N_-(q'_{02})
\end{equation} (cf. \ref{ineq}).

Now, for all $g\in \mathcal{H}'_1$, we have
\begin{eqnarray*}
q'_{01}[g] &=& \beta\int_{\mathbb{R}}|j'(s)|^2 ds - 2\frac{l^2}{\beta}\int_{\mathbb{R}}|j(s)|^2\left(\int_0^dV(s,u)\sin^2\frac{\pi l}{2d}udu\right)ds\\
&=&\beta\left(\int_{\mathbb{R}}|j'(s)|^2 ds - 2\int_{\mathbb{R}}V_*(s)|j(s)|^2ds\right),
\end{eqnarray*} 
where
$$
V_*(s) = \frac{l^2}{\beta}\int_0^dV(s,u)\sin^2\frac{\pi l}{2d}u\,du.
$$
Hence $q'_{01}$ represents the well studied one-dimensional self-adjoint operator
$$
-j''-V_*j
$$
on $L^2(\mathbb{R})$. The definitions of $\gamma_0$ and $\gamma_k$ imply that
$$
\gamma_0 = \int_{I_0}V_*(s)\,ds \mbox{ and } \gamma_k = \int_{I_k}|s|V_*(s)\,dx, \,k\not= 0.
$$
Thus there is a constant $ C_{14}>0$ such that
\begin{equation}\label{est14}
N_-(q'_{01}) \le 1+ C_{14}\sum_{\gamma_k >C_{11}}\sqrt{\gamma_k}.
\end{equation}
On the subspace $\mathcal{H}'_2$, partitioning $\Omega$ into rectangles $\{\Omega_k\}$ and following a similar procedure leading to \eqref{eq}, one can show that there is a constant $C_{15}>0$ such that
\begin{equation}\label{est15}
N_-(q'_{02}) \le C_{15}\|V\|_{L_1(J_k, L_{\psi}(I))} = C_{15}\mathcal{D}_k,
\end{equation}
where $C_{15} = l^2C_{10}$. Thus for all $C_{12}< \frac{1}{C_{15}}$, it follows similarly to \eqref{est3} that
\begin{equation}\label{est16}
N_-(q'_{02}) \le C_{15}\sum_{\mathcal{D}_k>C_{12}}\mathcal{D}_k.
\end{equation}
Hence \eqref{mr2} follows from \eqref{est13}, \eqref{est14} and \eqref{est16}, with $C_{13} = \max\{C_{14}, C_{15}\}$.

\end{proof}

\subsection{Estimates for the sum of eigenvalues}
We begin this section by stating the following result called the Lieb-Thirring inequality that we shall require in the proof of our estimate.
\begin{lemma}{\rm \cite[Theorem 4.38]{Rup}}\label{LT}
Let $\kappa > 0$ and let $\{\lambda_n(-\Delta - V)_-\}$ be a nondecreasing sequence of negative eigenvalues of the lower semi-bounded self-adjoint operator $-\Delta - V,\, V\ge 0 $. Then for $n\ge 2$ there is a constant $C_{16}>0$ depending on $\kappa$ and $n$  such that
\begin{equation}\label{est6}
\sum_n\lambda_n(-\Delta - V)^{\kappa}_-\le C_{10}(\kappa, n)\int_{\mathbb{R}^n}V(x)^{\kappa + \frac{n}{2}}\,dx
\end{equation} 
for any $V\in L^{\kappa +\frac{n}{2}}(\mathbb{R}^n)$. 
\end{lemma}
Similarly, let  $\{\lambda_n(Q)_-\}$ denote a nondecreasing sequence of negative eigenvalues of the operator in \eqref{wg}.
Then taking $\kappa =1$ and $n =2$, it follows from \eqref{est6} that
\begin{equation}\label{est7}
\sum_n\lambda_n(Q)_-\le C_{16}\|V\|^2_{L^2(\Omega)}, \,\,\forall V\ge 0.
\end{equation} 
 The following result can be found for example in \cite{Kato}.
 \begin{lemma}\label{Kato}
  Let  $T_1$ and $T_2$ be lower semi-bounded self-adjoint operators on a Hilbert space $\mathcal{H}$ and let $q_1$ and $a_2$ be their associated quadratic forms respectively. If $D(a_2)\subset D(a_1)$ and for all $u\in D(a_2)$, we have that \\$a_1[u] \le a_2[u]$, then $T_1 \le T_2$. Consequently, $\lambda_n(T_1)_- \le \lambda_n(T_2)_-$ and
 \begin{equation}\label{est8}
 \sum_n\lambda_n(T_1)_- \le \sum_n\lambda_n(T_2)_-.
 \end{equation}
 \end{lemma}
 
 \begin{theorem}\label{thm}
 Let $\{\lambda_n(Q')_-\}$ be a nondecreasing sequence of negative eigenvalues of the operator in \eqref{curvedop}. Then there is a constant $C_{11}>0$ depending on the curvature $k$ such that 
 \begin{equation}\label{est9}
 \sum_n\lambda_n(Q')_-\le C_{17}\|V\|^2_{L^2(\Omega)}, \,\,\forall V\ge 0.
 \end{equation}
 \end{theorem}
 \begin{proof}
 By \eqref{elliptic} and \eqref{newform}, we have that
\begin{eqnarray}\label{est10}
q'[f]&\le&\frac{1}{1-d\|k_+\|_{\infty}}\int_{\Omega}|\partial_sf|^2dsdu + (1+d\|k_-\|_{\infty})\int_{\Omega}|\partial_uf|^2dsdu \nonumber\\ &-&\frac{\pi^2}{4d^2}(1-d\|k_+\|_{\infty})\int_{\Omega}|f|^2dsdu - (1-d\|k_+\|_{\infty})\int_{\Omega}V|f|^2dsdu \nonumber \\
&\le& \max\left\{\frac{1}{1-d\|k_+\|_{\infty}},1+d\|k_-\|_{\infty}\right\}\left(\int_{\Omega}|\partial_sf|^2dsdu + \int_{\Omega}|\partial_uf|^2dsdu\right)\nonumber\\
&-&\frac{\pi^2}{4d^2}(1-d\|k_+\|_{\infty})\int_{\Omega}|f|^2dsdu - (1-d\|k_+\|_{\infty})\int_{\Omega}V|f|^2dsdu\nonumber\\
&=& M\left(\int_{\Omega}|\partial_sf|^2dsdu + \int_{\Omega}|\partial_uf|^2dsdu
-\frac{\pi^2}{4Md^2}(1-d\|k_+\|_{\infty})\int_{\Omega}|f|^2dsdu\right)\nonumber\\ &-&\frac{1}{M} (1-d\|k_+\|_{\infty})\int_{\Omega}V|f|^2dsdu,
\end{eqnarray}
where 
$$
M:=\max\left\{\frac{1}{1-d\|k_+\|_{\infty}},1+d\|k_-\|_{\infty}\right\}.
$$

The right hand side of \eqref{est10} gives a quadratic form associated with the lower semi-bounded self-adjoint operator
\begin{equation}\label{newoperator'}
P':=-\Delta - \lambda^*_1-V^*,\,\,V^*\ge 0
\end{equation}
on $L^2(\Omega, dsdu)$, where $\lambda^*_1= \frac{\pi^2}{4Md^2}(1-d\|k_+\|_{\infty})$ and $V^*= \frac{V(1-d\|k_+\|_{\infty})}{M}$.\\
If $p'$ is the quadratic form corresponding to \eqref{newoperator'}, then
$$
q'[f] \le p'[f].
$$
Consequently, \eqref{est8} gives
\begin{equation}\label{est11}
 \sum_n\lambda_n(Q')_- \le \sum_n\lambda_n(P')_-.
 \end{equation}
 Using \eqref{est7}, one has
 $$
 \sum_n\lambda_n(P')_- \le C_{16}\frac{(1-d\|k_+\|_{\infty})^2}{M^2}\|V\|^2_{L^2(\Omega)}.
 $$
 Hence \eqref{est9} follows with
 $$
 C_{17} = C_{16}\frac{(1-d\|k_+\|_{\infty})^2}{M^2}.
 $$
 \end{proof}
 
 In \eqref{est9}, the series on the left hand side converges for any $0\le V\in L^2(\Omega)$ with the curvature contributing greatly by inducing more eigenvalues than those in the straight case. Therefore the motion of the quantum particle is stable and the particle remains confined within a certain for a significant amount of time. Theorem \ref{thm} can easily be generalized by the dual-Thirring inequality for any $L^2$-orthonormal sequence in the domain of the quadratic form $q'$ where the optimal estimating constants are in a one-to-one correspondence, see for example {\rm \cite[Theorem 7.12]{Rup}}.
 \section{Conflicts of interest}
 The authors declare no conflicts of interest.

\end{document}